\documentclass[11pt,reqno]{amsart}
\usepackage{txfonts}
\usepackage{amsmath,amssymb,amsfonts,amsthm,mathrsfs}
\usepackage{graphics}
\usepackage{graphicx}
\usepackage{multirow}
\usepackage{caption2}
\usepackage{float}
\usepackage{booktabs}
 \newtheorem{theorem}{Theorem}[section]
 
 \newtheorem{proposition}[theorem]{Proposition}
 \newtheorem{lemma}[theorem]{Lemma}
 \newtheorem{definition}{Definition}[section]
 
 \newtheorem{assumption}{Assumption}
 \theoremstyle{remark}
 \newtheorem{remark}{Remark}

\allowdisplaybreaks
\begin{document}

\title[]{Convergence of multi-block
Bregman ADMM \\ for nonconvex  composite problems}

\date{}

\maketitle

\begin{center}
Fenghui~Wang, Wenfei~Cao, Zongben~Xu\footnote{Corresponding author (zbxu@mail.xjtu.edu.cn)}\\[0.2in]

\address{\small School of Mathematics and Statistics, Xi'an Jiaotong University, Xi'an, 710049, China }
\end{center}

\begin{abstract}
The alternating direction method with multipliers (ADMM) has been
one of most powerful and successful methods for solving  various
composite problems. The convergence of the conventional ADMM (i.e.,
2-block) for convex objective functions has been justified for a
long time, and its convergence for nonconvex objective functions
has, however, been established very recently. The multi-block ADMM,
a natural extension of ADMM, is a widely used scheme and has also
been found very useful in solving various nonconvex optimization
problems. It is thus expected to establish  convergence theory of
the multi-block ADMM under nonconvex frameworks. In this paper we
present a Bregman modification of 3-block ADMM and establish its
convergence for a large family of nonconvex functions. We further extend
the convergence results to the $N$-block case ($N \geq 3$),
which underlines the feasibility of multi-block ADMM applications in
nonconvex settings. Finally, we present a simulation study and a
real-world application to support the correctness of the obtained
theoretical assertions.\\

\end{abstract}

\section{Introduction}

Many problems arising in the fields of  signal \& image processing
and machine learning \cite{bpc,yl} involve finding a minimizer of
the sum of $N~(N\ge2)$ functions with linear equality constraint. If
$N=2$, the problem then consists of solving
\begin{align}\label{p11}
\min & \ f(x)+g(y)\nonumber\\
 \mathrm{s.t.} & \  Ax+By=0
\end{align}
where $A\in\mathrm{R}^{m\times n_1}$ and $B\in\mathrm{R}^{m\times
n_2}$ are given matrices,  $f:\mathrm{R}^{n_1}\to\mathrm{R}$  is a
proper lower semicontinuous function, and
$g:\mathrm{R}^{n_2}\to\mathrm{R}$ is a smooth function. Because of
its separable structure, problem (\ref{p11}) can be efficiently
solved by ADMM, namely, through the procedure
\begin{align}\label{admm}
\begin{split}
\left\{
  \begin{array}{ll}
   x^{k+1}=\arg\min\limits_{x\in\mathrm{R}^{n_1}} L_{\alpha}(x,y^{k},p^k)  \\
y^{k+1}=\arg\min\limits_{y\in\mathrm{R}^{n_2}} L_{\alpha}(x^{k+1},y,p^k)  \\
p^{k+1}=p^k+\alpha(Ax^{k+1}+By^{k+1})
  \end{array}
\right.
\end{split}
\end{align}
where $\alpha$ is a penalty parameter and
\begin{align*}
L_{\alpha}(x,y,p):=f(x)+g(y)+\langle p,
Ax+By\rangle+\frac\alpha2\|Ax+By\|^2
\end{align*}
is the associated augmented Lagrangian function with multiplier $p$.
So far, various variants of the conventional ADMM have been
suggested. Among such varieties, Bregman ADMM (BADMM) is the one
designed to improve the performance of procedure \eqref{admm}
\cite{fwb,wb,wang,zhang}.  More specifically, BADMM takes the
following iterative form:
\begin{align}\label{badmm}
\begin{split}
 x^{k+1}&=\arg\min\limits_{x\in\mathrm{R}^{n_1}} L_{\alpha}(x,y^{k},p^k)+\triangle_{\phi}(x,x^k) \\
y^{k+1}&=\arg\min\limits_{y\in\mathrm{R}^{n_2}} L_{\alpha}(x^{k+1},y,p^k)+\triangle_{\psi}(y,y^k) \\
p^{k+1}&=p^k+\alpha(Ax^{k+1}+By^{k+1}),
\end{split}
\end{align}
where $\triangle_{\phi}$ and $\triangle_{\psi}$ are the Bregman
distance with respect to functions $\phi$ and $\psi,$  respectively.

ADMM was introduced in the early 1970s \cite{gm2,gm},  and its
convergence properties for convex objective functions have been
extensively studied. The first convergent result was established for
strongly convex functions \cite{gm2,gm}, and then extended to
general convex functions \cite{eck,ecb}.  It has been shown that
ADMM can converge at a sublinear rate of  $\mathcal{O}(1/k)$
\cite{hy,ms}, and $\mathcal{O}(1/k^2)$ for the accelerated version
\cite{gds}. The convergence of BADMM for convex objective functions
has also been examined with the Euclidean distance \cite{ct},
Mahalanobis distance \cite{zhang}, and the general Bregman distance
\cite{zhang}.

Recently, there has been an increasing interest in the study of
ADMM for nonconvex objective functions. On one hand,  the ADMM
algorithm is highly successful in solving various nonconvex examples
ranging from nonnegative matrix factorization, distributed matrix
factorization, distributed clustering, sparse zero variance
discriminant analysis, polynomial optimization, tensor
decomposition, to matrix completion (see e.g.
\cite{hcw,ls,xyw,zk,zy}). On the other hand, the convergence
analysis of nonconvex ADMM
 is generally very difficult, due to the failure of the F\'{e}jer
monotonicity of iterates. In \cite{hlr}, the subsequential
convergence of ADMM for general nonconvex functions has been proved.
Furthermore, the global convergence of ADMM for certain type of
nonconvex functions has been proved in \cite{lp,wxx}.

The purpose of the present study is to examine convergence of ADMM
with 3 blocks (i.e., $N=3$). The obtained results then can naturally
be generalized  to the case of ADMM with multiple blocks. Thus, in
the present paper we first consider the following 3-block composite
optimization problem:
\begin{align}\label{p1}
\min & \ f(x)+g(y)+h(z)\nonumber\\
 \mathrm{s.t.} & \  Ax+By+Cz=0
\end{align}
where $A\in\mathrm{R}^{m\times n_1},B\in\mathrm{R}^{m\times n_2}$
and $C\in\mathrm{R}^{m\times n_3}$ are given matrices,
$f:\mathrm{R}^{n_1}\to\mathrm{R},g:\mathrm{R}^{n_2}\to\mathrm{R}$
are proper lower semicontinuous functions, and
$h:\mathrm{R}^{n_3}\to\mathrm{R}$ is a smooth function. To solve
such a problem, it is natural to extend the ADMM to the following
form:
 \begin{align}\label{B5}
\begin{split}
\left\{
   \begin{array}{ll}
x^{k+1}=\arg\min\limits_{x\in\mathrm{R}^{n_1}} L_{\alpha}(x,y^{k},z^k,p^k)\\
y^{k+1}=\arg\min\limits_{y\in\mathrm{R}^{n_2}} L_{\alpha}(x^{k+1},y,z^k,p^k)   \\
z^{k+1}=\arg\min\limits_{z\in\mathrm{R}^{n_3}} L_{\alpha}(x^{k+1},y^{k+1},z,p^k)    \\
p^{k+1}=p^k+\alpha(Ax^{k+1}+By^{k+1}+Cz^{k+1})
   \end{array} \right.
\end{split}
\end{align}
where the augmented Lagrangian function
$L_{\alpha}:\mathrm{R}^{n_1}\times\mathrm{R}^{n_2}
 \times\mathrm{R}^{n_3}\times\mathrm{R}^{m}\to  \mathrm{R}$ is defined by
\begin{align}\label{lg}
L_{\alpha}(x,y,z,p):=f(x)+g(y)+h(z)+\langle p,
Ax+By+Cz\rangle+\frac\alpha2\|Ax+By+Cz\|^2.
\end{align}
Unlike the conventional  ADMM with 2 blocks, the convergence of
algorithm \eqref{B5}, called the 3-block ADMM henceforth, has
remained unclear even for convex objective functions. Although it is
not necessarily convergent in general \cite{ch}, the 3-block ADMM
does converge under some restrictive conditions; for example, under
the strong convexity condition of all objective functions  (see e.g.
\cite{hy}). Recently,  Li, Sun, and Toh \cite{lst} proposed a
modification of algorithm \eqref{B5}, called the semi-proximal
3-block ADMM as follows
 \begin{align}\label{B6}
\begin{split}
\left\{
   \begin{array}{ll}
x^{k+1}=\arg\min\limits_{x\in\mathrm{R}^{n_1}} L_{\alpha}(x,y^{k},z^k,p^k)+\frac12\|x-x^k\|^2_{T_1}\\
y^{k+1}=\arg\min\limits_{y\in\mathrm{R}^{n_2}} L_{\alpha}(x^{k+1},y,z^k,p^k)+\frac12\|y-y^k\|^2_{T_2}   \\
z^{k+1}=\arg\min\limits_{z\in\mathrm{R}^{n_3}} L_{\alpha}(x^{k+1},y^{k+1},z,p^k)+\frac12\|z-z^k\|^2_{T_3}    \\
p^{k+1}=p^k+\alpha(Ax^{k+1}+By^{k+1}+Cz^{k+1})
   \end{array} \right.
\end{split}
\end{align}
where $\|\cdot\|_{T_i}$ denotes ellipsoidal norms, $i=1,2,3$. They
proved the convergence of the algorithm when $f,g,h$ are all convex
and one of them is at least strongly convex.

Motivated by Bregman ADMM, we propose to use the following 3-block
Bregman ADMM for solving the optimization problem  \eqref{p1}:
 \begin{align}\label{B7}
\begin{split}
\left\{
   \begin{array}{ll}
x^{k+1} =\arg\min\limits_{x\in\mathrm{R}^{n_1}} L_{\alpha}(x,y^{k},z^k,p^k)+\triangle_{\phi}(x,x^k)  \\
y^{k+1} =\arg\min\limits_{y\in\mathrm{R}^{n_2}} L_{\alpha}(x^{k+1},y,z^k,p^k)+\triangle_{\psi}(y,y^k) \\
z^{k+1} =\arg\min\limits_{y\in\mathrm{R}^{n_3}} L_{\alpha}(x^{k+1},y^{k+1},z,p^k)+\triangle_{\varphi}(z,z^k)  \\
p^{k+1} =p^k+\alpha(Ax^{k+1}+By^{k+1}+Cz^{k+1})
 \end{array} \right.
\end{split}
\end{align}
where, as mentioned before, $\triangle_{\phi}, \triangle_{\psi}$ and
$ \triangle_{\varphi}$ are the Bregman distance associated with
functions $\phi,\psi,$ and $\varphi$,  respectively. In the present
paper, our aim is to justify the convergence of 3-block BADMM under
nonconvex frameworks. We will show that the 3-block BADMM can
converge if the objective function is subanalytic and matrix $C$ has
full-row rank.

\section{Preliminaries}

In what follows, $\mathrm{R}^n$ will stand for the $n$-dimensional
Euclidean space,
$$\langle x,y\rangle=x^{\top}y=\sum _{i=1}^n
x_iy_i, \ \|x\|=\sqrt{\langle x, x\rangle},$$ where
$x,y\in\mathrm{R}^n$ and $\top$ stands for the transpose operation.

\subsection{Subdifferentials}

Given a function $f:\mathrm{R}^n\to \mathrm{R}$ we denote by
$\mathrm{dom}f$ the domain of $f$, namely, $\mathrm{dom}f:=\{x\in
\mathrm{R}^n: f(x)<+\infty \}$. A function $f$ is said to be proper
if $\mathrm{dom}f\neq\emptyset;$ lower  semicontinuous at the point
$x_0$ if
$$\liminf_{x\to x_0}f(x)\ge f(x_0).$$
 If $f$ is lower semicontinuous at every point of its domain of definition, then it is simply
 called a lower semicontinuous function.

\begin{definition}
Let $f:\mathrm{R}^n\to \mathrm{R}$ be a proper lower semi-continuous
function.
\begin{itemize}
  \item[(i)] Given $x\in \mathrm{dom} f,$ the Fr\'{e}chet subdifferential of $f$ at $x$,
  written by $\widehat{\partial} f(x)$, is the set of  all elements $u\in \mathrm{R}^n$ which satisfy
  \begin{align*}
\lim_{y\neq x}\inf_{y\to x}\frac{f(y)-f(x)-\langle u,
y-x\rangle}{\|x-y\|}\ge0.
  \end{align*}
  \item[(ii)] The limiting  subdifferential, or simply subdifferential, of $f$ at $x$,
  written by $\partial f(x)$, is defined as
  \begin{align*}
\partial f(x)=\{u\in \mathrm{R}^n: \exists x^k\to x, f(x^k)\to f(x),\\
u^k\in\widehat{\partial} f(x^k)\to u, k\to\infty\}.
  \end{align*}
\item[(iii)]
A critical point or stationary point of $f$ is a point $x^*$ in the
domain of $f$ satisfying $0\in\partial f(x^*).$\end{itemize}
\end{definition}

\begin{definition}
An element $w^*:=(x^{*}, y^*,z^*, p^*)$ is  called a critical point
or stationary point of the Lagrangian function $L_{\alpha}$ defined
as in \eqref{lg} if it satisfies:
\begin{align}
\left\{
\begin{array}{l}
A^{\top}p^*\in-\partial f(x^*), \ B^{\top}p^*\in -\partial g(y^{*}), \\
C^{\top}p^*=-\nabla h(z^*), \ Ax^*+By^*+Cz^*=0.
  \end{array}
\right.
\end{align}
\end{definition}

The existence of proper lower semicontinuous functions and
properties of subdifferential can see~\cite{mor}. We particularly
collect the following basic properties of the subdifferential.
\begin{proposition}
Let $f:\mathrm{R}^n\to \mathrm{R}$ and $g:\mathrm{R}^n\to
\mathrm{R}$ be proper lower semi-continuous functions. Then the
following holds:
\begin{itemize}
  \item[(i)]   $\widehat{\partial} f(x)\subset  \partial  f(x)$ for each $x\in \mathrm{R}^n.$
  Moreover, the first set is closed and convex, while the second is closed, and not necessarily convex.
  \item[(ii)]  Let $(u^k, x^k)$ be sequences such that $x^k\to x, u^k\to u, f(x^k)\to f(x)$ and $u^k\in \partial f(x^k).$
  Then $u\in \partial f(x).$
  \item[(iii)] The Fermat's rule remains true: if $x_0\in\mathrm{R}^n$ is a local minimizer of $f$, then $x_0$ is a
  critical point or stationary point of $f$, that is, $0\in \partial f(x_0).$
  \item[(iv)] If $f$ is continuously differentiable function, then $\partial (f+g)(x)=\nabla f(x)+ \partial g(x).$
\end{itemize}
\end{proposition}

 A function $f$ is said to be
 {\em  $\ell_{f}$-Lipschitz continuous}  $(\ell_{f}\ge 0)$ if
 $$\|f(x)-f(y)\|\le\ell_{f}\|x-y\|,$$
 for any $x,y\in\mathrm{dom}f$; {\em  $\mu$-strongly convex} $(\mu>0)$ if
 \begin{align}\label{B1}
 f(y)\ge f(x)+\langle\xi(x), y-x\rangle+\frac\mu2\|y-x\|^2,
 \end{align}
 for any $x,y\in\mathrm{dom}f$ and $\xi(x)\in\partial f(x);$ {\em
 coercive}  if
 \begin{align}
 \lim_{\|x\|\to\infty}f(x)=+\infty.
 \end{align}

\subsection{Kurdyka-{\L}ojasiewicz inequality}\label{s1}

The Kurdyka-{\L}ojasiewicz (K-L) inequality was first introduced by
{\L}ojasiewicz \cite{loj} for  real analytic functions, and then was
extended by Kurdyka \cite{kur} to smooth functions whose graph
belongs to an o-minimal structure. Recently, this notion was further
extended for  nonsmooth subanalytic functions \cite{bdo}.

\begin{definition}[K-L inequality]
A function $f:\mathrm{R}^n\to \mathrm{R}$ is said to satisfy the
 K-L inequality at $x_0$ if
 there exists $\eta>0, \delta>0, \varphi\in\mathscr{A}_{\eta}$, such that for all
$x\in\mathcal{O}(x_0,\delta)\cap\{x: f(x_0)<f(x)<f(x_0)+\eta\}$
\begin{align*}
 \varphi'(f(x)-f(x_0))\mathrm{dist}(0, \partial f(x))\ge1,
\end{align*}
where $\mathrm{dist}(x_0, \partial f(x)):=\inf\{\|x_0-y\|:y\in
\partial f(x)\},$ and
 $\mathscr{A}_{\eta}$ stand for the class of functions $\varphi:[0,\eta)\to
\mathrm{R}^{+}$ with the properties: (a) $\varphi$ is  continuous on
$[0,\eta)$; (b) $\varphi$ is smooth concave on $(0,\eta)$; (c)
 $\varphi(0)=0, \varphi'(x)>0, \forall x\in (0,\eta)$.
\end{definition}

The following is an extension of the conventional K-L inequality
\cite{bst}.

\begin{lemma}[K-L inequality on compact subsets]\label{KL}
Let $f:\mathrm{R}^n\to \mathrm{R}$ be a proper lower semi-continuous
function and let $\Omega\subseteq \mathrm{R}^n$ be a compact set. If
$f$ is a constant on $\Omega$ and $f$ satisfies the K-L inequality
at each point in $\Omega$, then
 there exists $\eta>0, \delta>0, \varphi\in\mathscr{A}_{\eta}$, such that  for all $x_0\in\Omega$ and for all
$x\in\{x\in\mathrm{R}^n:\mathrm{dist}(x,\Omega)<\delta)\}\cap\{x\in\mathrm{R}^n:
f(x_0)<f(x)<f(x_0)+\eta\}$,
\begin{align*}
 \varphi'(f(x)-f(x_0))\mathrm{dist}(0, \partial f(x))\ge1.
\end{align*}
\end{lemma}

Typical functions satisfying the K-L inequality include strongly
convex functions, real analytic functions, semi-algebraic functions
and subanalytic functions.

A subset $C\subset \mathrm{R}^n$ is said to be {\em semi-algebraic}
if it can be written as
  \begin{align*}
   C=\bigcup_{j=1}^r\bigcap_{i=1}^{s}\{x\in \mathrm{R}^n: g_{i,j}(x)=0, h_{i,j}(x)<0\},
  \end{align*}
  where $g_{i,j},h_{i,j}:\mathrm{R}^n\to\mathrm{R}$ are real polynomial functions.
Then a function $f:\mathrm{R}^n\to \mathrm{R}$ is called {\em
semi-algebraic}  if its graph
  \begin{align*}
  \mathcal{G}(f):=\{(x,y)\in\mathrm{R}^{n+1}:f(x)=y\}
  \end{align*}
  is a semi-algebraic subset in $\mathrm{R}^{n+1}$. For example, the $\ell_{q}$ norm
  $\|x\|_{q}:=\sum_i|x_i|^{q}$ with $0<q\le1$, the sup-norm $\|x\|_{\infty}:=\max_i|x_i|,$
  the Euclidean norm $\|x\|$, $\|Ax-b\|^{q}_{q}$,
  $\|Ax-b\|$ and $\|Ax-b\|_\infty$ are all semi-algebraic functions for any matrix $A$ \cite{bst,yin}.

A real function on  $\mathrm{R}$ is said to be {\em analytic} if it
possesses derivatives of all orders
 and agrees with its Taylor series in a neighborhood of every point. For a real function $f$  on
 $\mathrm{R}^n$, it is said to be {\em analytic} if the function of one variable
 $g(t):=f(x+ty)$ is analytic for any $x,y\in\mathrm{R}^n$.
It is readily seen that  real polynomial functions such as quadratic
functions
  $\|Ax-b\|^2$  are analytic. Moreover, the $\varepsilon$-smoothed $\ell_{q}$ norm
  $\|x\|_{\varepsilon,q}:=\sum_i(x_i^2+\varepsilon)^{q/2}$ with $0<q\le1$ and the logistic
  loss function $\log(1+e^{-t})$ are all examples for real analytic functions \cite{yin}.

A subset $C\subset \mathrm{R}^n$ is said to be {\em subanalytic} if
it can be written as
  \begin{align*}
   C=\bigcup_{j=1}^r\bigcap_{i=1}^{s}\{x\in \mathrm{R}^n: g_{i,j}(x)=0, h_{i,j}(x)<0\},
  \end{align*}
  where $g_{i,j},h_{i,j}:\mathrm{R}^n\to\mathrm{R}$ are real analytic functions.
Then a function $f:\mathrm{R}^n\to \mathrm{R}$ is called {\em
subanalytic} if its graph $ \mathcal{G}(f)$ is a subanalytic subset
in $\mathrm{R}^{n+1}$. It is clear that both real analytic and
semi-algebraic functions are subanalytic. Generally speaking, the
sum of two subanalytic functions is not necessarily subanalytic. It
is known, however, that for two subanalytic functions, if at least
one function maps bounded sets to bounded sets, then their sum is
also subanalytic, as shown in \cite{bdo,yin}. In particular, the sum
of a subanalytic function and a analytic function is subanalytic.
Some subanalytic functions that are widely used are as follows:
\begin{itemize}
\item[(i)]  $\|Ax-b\|^2+\lambda\|y\|^q_q $;
\item[(ii)]  $\|Ax-b\|^2+\lambda\sum_i(y^2_i+\varepsilon)^{q/2}$;
  \item[(iii)]
  $\frac1n\sum_{i=1}^n\log(1+\exp(-c_i(a_i^{\top}x+b))+\lambda\|y\|^q_q$;
\item[(iv)]
  $\frac1n\sum_{i=1}^n\log(1+\exp(-c_i(a_i^{\top}x+b))+\lambda\sum_i(y^2_i+\varepsilon)^{q/2}$.
\end{itemize}

\subsection{Bregman distance}
The  Bregman distance, first introduced  in 1967 \cite{berg}, plays
an important role in various iterative algorithms. As a
generalization of squared Euclidean distance, the Bregman distance
share many similar nice properties of the Euclidean distance.
However, the Bregman distance is  not a real metric, since it does
not satisfy the
 triangle inequality nor symmetry.
 For a convex differential function $\phi$, the associated Bregman distance is defined as
\begin{align*}
\triangle_{\phi}(x,y)=\phi(x)-\phi(y)-\langle\nabla\phi(y),x-y\rangle.
\end{align*}
In particular, if we let $\phi(x):=\|x\|^2$ in the above, then it is
reduced to $\|x-y\|^2$, namely, the classical Euclidean distance.
Some nontrivial examples of Bregman distance include \cite{Ban}:
\begin{itemize}
\item[(i)]  Itakura-Saito distance:  $\sum_ix_i(\log
  x_i/y_i)- \sum_i(x_i-y_i)$;
  \item[(ii)]  Kullback-Leibler divergence: $\sum_ix_i(\log
  x_i/y_i)$;
\item[(iii)]  Mahalanobis distance: $\|x-y\|^2_Q=\langle Qx, x\rangle$
  with $Q$ a symmetric positive definite matrix.
\end{itemize}

The following proposition collects some useful properties of Bregman
distance.

 \begin{proposition}\label{pro}
Let $\phi$ be a convex differential function and
$\triangle_{\phi}(x,y)$ the associated  Bregman distance.
\begin{itemize}
  \item[(i)]  Non-negativity: $\triangle_{\phi}(x,y)\ge0,\triangle_{\phi}(x,x)=0$ for all $x,y$.
  \item[(ii)]  Convexity: $\triangle_{\phi}(x,y)$ is convex in $x$, but not necessarily in $y$.
  \item[(iii)]  Strong Convexity: If $\phi$ is $\delta$-strongly convex, then $\triangle_{\phi}(x,y)
  \ge\frac{\delta}{2}\|x-y\|^2$ for all $x,y$.
\end{itemize}
 \end{proposition}

\subsection{Basic assumptions}
In the research of present paper, we will make the following
assumptions:
\begin{assumption}\label{j1}
We assume that functions $f,g,h,C,\phi, \psi, \varphi$ in problem
\eqref{p1} have the following properties:
\begin{itemize}
  \item[(a1)] $\langle CC^{\intercal}x, x\rangle=\|x\|^2_{C^{\intercal}}\ge
  \sigma_C\|x\|^2, \forall x\in\mathrm{R}^m$, namely, $C$ is full row rank;
  \item[(a2)] $\nabla h, \nabla \phi, \nabla\psi,  \nabla\varphi$ are
  Lipshitz continuous;
  \item[(a3)] either $f$ or $\phi$,
  either $g$ or $\psi$, and either $h$ or $\varphi$  are strongly convex;
  \item[(a4)] $f+g+h$ is subanalytic,
\end{itemize}
where $\sigma_C$ and $\ell_h$ are both positive real numbers.
\end{assumption}

In implementation of BADMM \eqref{B7}, the parameter $\alpha$, and
the smooth convex functions $\phi$, $\psi,$ and $\varphi$ should be
regularized. We further assume $\textbf{Assumption 2}$:
\begin{align}\label{C4}
\alpha>
  \frac{4[(\ell_{h}+\ell_{\varphi})^2+\ell_{\varphi}^2]}{\mu_3\sigma_C},
\end{align}
where $\mu_3$ is the strong convexity coefficient of $h$ or
$\varphi$, and $\ell_h$ and $\ell_{\varphi}$ are respectively the
Lipschitz coefficient of $\nabla h$ and $\nabla \varphi$.

We remark that conditions (a1)-(a2) above are standard assumptions
even for convex settings. Condition (a3) is used to guarantee the
sufficient descent property of iterates, and condition (a4) is a
basic assumption assuring that the function $\hat{L}$, to be defined
in the next section, can satisfy the K-L inequality, which in turn
will imply the global convergence of the proposed algorithm.

The smooth convex functions in the Bregman distance  are very easily
specified; for example, take $\phi(\cdot)=\psi(\cdot)=
\varphi(\cdot)=\frac{1}{2}\|\cdot\|^2.$ Note that if $\phi$ is
$\mu_1$-strongly convex, then its Bregman distance satisfies
\begin{align}\label{B8}
\triangle_\phi(x,y)\ge\frac{\mu_1}{2}\|x-y\|^2,
\end{align}
 which follows from Proposition \ref{pro}.

\section{Convergence Analysis}
In this section, under the Assumptions 1 and 2 we firstly give a convergence result for
the BADMM  with 3-block procedure~\eqref{B7}, and then extend this result to the $N$-block ($N \geq 3$) case. The main results are presented in the subsection~\ref{mainRes}.

For convenience, we first fix the following notations:
\begin{align*}
&\sigma_0=\frac{2\ell_{\varphi}^2}{\alpha\sigma_C}, \
\sigma_1=\frac12\min \left(\mu_1, \mu_2,
\mu_3-\frac{4(\ell_{h}+\ell_{\varphi})^2}{\alpha\sigma_C}-
  \frac{4\ell_{\varphi}^2}{\alpha\sigma_C}\right),\\
&u=(x,y,z), w=(x,y,z,p), \hat{w}=(x,y,z,p,\hat{z}),\\
&u^k=(x^{k},y^{k},z^{k}), w^k=(x^{k},y^{k},z^{k},p^k), \hat{w}^k=(x^{k},y^{k},z^{k},p^{k},z^{k-1}),\\
&\|w\|=(\|x\|^2+\|y\|^2+\|z\|^2)^{1/2}, \|w\|_1=\|x\|+\|y\|+\|z\|,
\end{align*}
where  $\mu_1$   is the strong convexity coefficient of $f$ or
$\phi$, and $\mu_2$   is the strong convexity coefficient of $g$ or
$\psi$.
 Clearly both $\sigma_0$ and $\sigma_1$ are positive by our assumptions.
 Also, we define a new function $\hat{L}:\mathrm{R}^{n_1}\times\mathrm{R}^{n_2}
 \times\mathrm{R}^{n_3}\times\mathrm{R}^{m}\times\mathrm{R}^{n_3}\to  \mathrm{R}$ by
\begin{align}\label{eq1}
\hat{L}(\hat{w})=L_\alpha(w)+\sigma_0\|z-\hat{z}\|^2.
\end{align}

\subsection{Some lemmas}
We establish a series of lemmas to support the proof of convergence
of BADMM  with 3-block procedure~\eqref{B7}.
\begin{lemma}\label{Le1}
 For each $k\in\mathrm{N}$
\begin{align}\label{D1}
\|p^{k+1}-p^{k}\|^2&\le
\frac{2(\ell_{h}+\ell_{\varphi})^2}{\sigma_C} \|z^{k+1}-z^{k}\|^2 +
\frac{2\ell_{\varphi}^2}{\sigma_C}\|z^{k}-z^{k-1}\|^2.
\end{align}
\end{lemma}

\begin{proof}
By our assumptions on $C$, we have
\begin{align}\label{D5}
\|C^{\intercal}(p^{k+1}-p^{k})\|^2 =\langle
CC^{\intercal}(p^{k+1}-p^{k}), p^{k+1}-p^{k}\rangle\ge\sigma_C
\|p^{k+1}-p^{k}\|^2.
\end{align}
Applying  Fermat's rule to $z$-subproblem in \eqref{B7}, we then get
\begin{align*}
 \nabla h(z^{k+1})+ C^{\intercal}(p^k+\alpha (Ax^{k+1}+By^{k+1}+Cz^{k+1}))
 +\nabla\varphi(z^{k+1})-\nabla\varphi(z^k)=0.
\end{align*}
Note that $p^{k+1}=p^k+\alpha (Az^{k+1}+By^{k+1}+Cz^{k+1})$. It then
follows that
\begin{align}\label{D2}
 \nabla h(z^{k+1})+C^{\intercal}p^{k+1}
 +\nabla\varphi(z^{k+1})-\nabla\varphi(z^k)=0,
\end{align}
so that
\begin{align*}
&\quad \ \|C^{\intercal}(p^{k+1}-p^{k})\|^2\\
&=\|\nabla h(z^{k+1})-\nabla h(z^{k}) +(\nabla\varphi(z^{k+1})-\nabla\varphi(z^k))+(\nabla\varphi(z^{k-1})-\nabla\varphi(z^k))\|^2\\
&\le\big(\|\nabla h(z^{k+1})-\nabla h(z^{k})\| +\|\nabla\varphi(z^{k+1})-\nabla\varphi(z^k)\|+\|\nabla\varphi(z^{k-1})-\nabla\varphi(z^k)\|\big)^2\\
&\le\big(\ell_{h}\|z^{k+1}-z^{k}\|+\ell_{\varphi}\|z^{k}-z^{k+1}\|+\ell_{\varphi}\|z^{k}-z^{k-1}\|\big)^2\\
&\le2(\ell_{h}+\ell_{\varphi})^2\|z^{k+1}-z^{k}\|^2+2\ell_{\varphi}^2\|z^{k}-z^{k-1}\|^2.
\end{align*}
This together with \eqref{D5} at once yields inequality \eqref{D1}.
\end{proof}

\begin{lemma}\label{Le2}
For each $k\in\mathrm{N}$
\begin{align}\label{F1}
\begin{split}
L_{\alpha}(w^{k+1}) &\le
L_{\alpha}(w^{k})+\left(\frac{2(\ell_{h}+\ell_{\varphi})^2}{\alpha\sigma_C}-\frac{\mu_3}{2}\right)
\|z^{k+1}-z^{k}\|^2\\
&\quad + \frac{2\ell_{\varphi}^2}{\alpha\sigma_C}\|z^{k}-z^{k-1}\|^2
-\frac{\mu_1}{2}\|x^{k+1}-x^k\|^2-\frac{\mu_2}{2}\|y^{k+1}-y^k\|^2.
\end{split}
\end{align}
\end{lemma}

\begin{proof}
First we show that if either $f$ or $\phi$ is  strongly convex, then
it follows that
\begin{align}\label{eq2}
L_{\alpha}(x^{k+1},y^{k},z^{k},p^{k})&\le
L_{\alpha}(x^{k},y^{k},z^{k},p^{k})-\frac{1}{2\mu_1}\|x^{k+1}-x^k\|^2.
\end{align}
In fact, if $f$ is strongly convex, then
$L_\alpha(x,y^k,z^k,p^k)+\triangle_\phi(x,x^k)$ is  strongly convex
with modulus $\mu_1$, and thus inequality \eqref{eq2} follows from
\eqref{B1}. Let us now justify the case whenever $\phi$ is strongly
convex. As $y^{k+1}$ is a minimizer of
$L_\alpha(x,y^k,z^k,p^k)+\triangle_\phi(x,x^k)$, we have
\begin{align*}
L_\alpha(x^{k+1},y^k,z^k,p^k) &\le L_\alpha(x,y^k,z^k,p^k)-\triangle_\phi(x^{k+1},x^k) \\
   &\le L_\alpha(x,y^k,z^k,p^k)-\frac{1}{2\mu_1}\|x^{k+1}-x^k\|^2,
\end{align*}
where the last inequality follows from \eqref{B8}. Similarly, we
have
\begin{align*}
 L_{\alpha}(x^{k+1},y^{k+1},z^{k},p^{k})&\le L_{\alpha}(x^{k+1},y^{k},z^{k},p^{k})-\frac{1}{2\mu_2}\|y^{k+1}-y^k\|^2\\
 L_{\alpha}(x^{k+1},y^{k+1},z^{k+1},p^{k})&\le L_{\alpha}(x^{k+1},y^{k+1},z^{k},p^{k})-\frac{1}{2\mu_3}\|z^{k+1}-z^k\|^2,
\end{align*}
and from the last equality in \eqref{B7} we have
\begin{align*}
 L_{\alpha}(x^{k+1},y^{k+1},z^{k+1},p^{k+1})&= L_{\alpha}(x^{k+1},y^{k+1},z^{k+1},p^{k})+\frac1\alpha\|p^{k+1}-p^{k}\|^2.
\end{align*}
Adding up the above formulas, we get
\begin{align}\label{B3}
\begin{split}
 L_{\alpha}(w^{k+1})
&\le L_{\alpha}(w^{k})+\frac1\alpha\|p^{k+1}-p^{k}\|^2 -\frac{1}{2\mu_1}\|x^{k+1}-x^k\|^2\\
&\quad-\frac{1}{2\mu_2}\|y^{k+1}-y^k\|^2-\frac{1}{2\mu_3}\|z^{k+1}-z^k\|^2.
\end{split}
\end{align}
This together with \eqref{D1} yields inequality \eqref{F1} as
desired.
\end{proof}


\begin{lemma}\label{Le3}
For each $k\in\mathrm{N}$
\begin{align*}
\hat{L}(\hat{w}^{k+1})\le\hat{L}(\hat{w}^{k})
-\sigma_1(\|x^{k+1}-x^k\|^2+\|y^{k+1}-y^k\|^2+\|z^{k}-z^{k+1}\|^2).
\end{align*}
\end{lemma}

\begin{proof}
It follows from lemmas \ref{Le1} and \ref{Le2} that
\begin{align*}
&L_{\alpha}(x^{k+1},y^{k+1},z^{k+1},p^{k+1})
  -L_{\alpha}(x^{k},y^{k},z^{k},p^{k})\\
&\le\left(\frac{2(\ell_{h}+\ell_{\varphi})^2}{\alpha\sigma_C}-\frac{\mu_3}{2}\right)
\|z^{k+1}-z^{k}\|^2 + \frac{2\ell_{\varphi}^2}{\alpha\sigma_C}\|z^{k}-z^{k-1}\|^2\\
&\quad
-\frac{\mu_1}{2}\|x^{k+1}-x^k\|^2-\frac{\mu_2}{2}\|y^{k+1}-y^k\|^2,
\end{align*}
which implies
\begin{align*}
&L_{\alpha}(x^{k+1},y^{k+1},z^{k+1},p^{k+1})+\sigma_0\|z^{k+1}-z^k\|^2\\
&\le L_{\alpha}(x^{k},y^{k},z^{k},p^{k})+\sigma_0\|z^{k}-z^{k-1}\|^2\\
&\quad-\left(\frac{\mu_3}{2}-\frac{2(\ell_{h}+\ell_{\varphi})^2}{\alpha\sigma_C}-
  \frac{2\ell_{\varphi}^2}{\alpha\sigma_C}\right)\|z^{k}-z^{k+1}\|^2\\
  &\quad -\frac{\mu_1}{2}\|x^{k+1}-x^k\|^2-\frac{\mu_2}{2}\|y^{k+1}-y^k\|^2\\
  &\le L_{\alpha}(x^{k},y^{k},z^{k},p^{k})+\sigma_0\|z^{k}-z^{k-1}\|^2\\
  &\quad-\sigma_1(\|x^{k+1}-x^k\|^2+\|y^{k+1}-y^k\|^2+\|z^{k}-z^{k+1}\|^2).
\end{align*}
Then lemma \ref{Le3} follows from our notations.
\end{proof}

\begin{lemma}\label{Le4}
If the sequence $\{u^k\}$ is bounded, then we have
$$\sum_{k=0}^{\infty}\|w^k-w^{k+1}\|^2<\infty.$$
In particular, the sequence $\|w^k-w^{k+1}\|$ is asymptotically
regular, namely, $\|w^k-w^{k+1}\|\to0$ as $k\to\infty.$ Moreover,
any cluster point of $w^k$ is a stationary point of the augmented
Lagrangian function $L_{\alpha}$ defined as in \eqref{lg}.
\end{lemma}

\begin{proof}
We first show that the sequence $\{w^k\}$ is bounded. Indeed we
deduce from Eq. \eqref{D2} that
\begin{align*}
\|C^{\intercal}p^k\|^2&=\|\nabla h(z^k)+\nabla\varphi(z^k)-\nabla\varphi(z^{k-1})\|^2\\
&\le(\|\nabla h(z^k)\|+\ell_{\varphi}\|z^k-z^{k-1}\|)^2\\
&\le2(\|\nabla h(z^k)\|^2+\ell_{\varphi}^2\|z^k-z^{k-1}\|^2).
\end{align*}
Since $C$ has full row rank, we have
\begin{align}\label{B2}
\sigma_C\|p^k\|^2\le2(\|\nabla
h(z^k)\|^2+\ell_{\varphi}^2\|z^k-z^{k-1}\|^2).
\end{align}
Note that  $\{u^k\}$ is bounded. This implies that the sequence
$\{p^k\}$ is bounded and so are the sequences $\{w^k\}$ and
$\{\hat{w}^k\}$.

 Since $\hat{w}^k$ is
bounded, there exists a subsequence $\hat{w}^{k_j}$ so that it is
convergent to some element $\hat{w}^{*}$. By our hypothesis, the
function $\hat{L}$ is lower semicontinuous, which leads to
$$\liminf_{j\to\infty}\hat{L}(\hat{w}^{k_j})\ge \hat{L}(\hat{w}^*),$$
so that $\hat{L}(\hat{w}^{k_j})$ is bounded from below. By Lemma
\ref{Le3}, $\hat{L}(\hat{w}^{k})$ is nonincreasing, so that
$\hat{L}(\hat{w}^{k_j})$ is a convergent sequence. Moreover
$\hat{L}(\hat{w}^{k})$ is also convergent and
$\hat{L}(\hat{w}^{k})\ge
 \hat{L}(\hat{w}^{*})$ for each $k$.

Now fix $k\in\mathrm{N}.$ By Lemma \ref{Le3}, we have
\begin{align*}
& \sigma_1\sum_{i=1}^{k}(\|x^{k+1}-x^k\|^2+\|y^{k+1}-y^k\|^2+\|z^{k}-z^{k+1}\|^2)\\
&\le\sum_{i=1}^{k} \hat{L}(\hat{w}^i)-\hat{L}(\hat{w}^{i+1})
=\hat{L}(\hat{w}^{1})-\hat{L}(\hat{w}^{k+1})\\
&\le\hat{L}(\hat{w}^{1})-\hat{L}(\hat{w}^{*})<\infty.
\end{align*}
Moreover, by inequality \eqref{D1}, we see that
$\sum_{k=0}^{\infty}\|p^k-p^{k+1}\|^2<\infty.$ This implies
$\sum_{k=0}^{\infty}\|w^k-w^{k+1}\|^2<\infty$, and  hence
$\|w^k-w^{k+1}\|\to0$.

Let $w^*=(x^*,y^*,z^*,p^*)$ be any cluster point of $w^k$ and let
$w^{k_j}$ be a subsequence of $w^k$  converging to $w^*$.
 It then  follows from algorithm \eqref{B7}  that
\begin{align*}
p^{k+1}&=p^{k}+\alpha (Ax^{k+1}+By^{k+1}+Cz^{k+1}),\\
-\partial f(x^{k+1})&  \ni A^{\intercal}p^{k}+\alpha A^{\intercal}(Ax^{k+1}+By^k+Cz^k)+\nabla\phi(x^{k+1})-\nabla\phi(x^k)\\
  &  = A^{\intercal}p^{k+1}+\alpha A^{\intercal}B(y^k-y^{k+1})+\alpha A^{\intercal}C(z^k-z^{k+1})+\nabla\phi(x^{k+1})-\nabla\phi(x^k),\\
  -\partial g(y^{k+1})&  \ni B^{\intercal}p^{k}+\alpha B^{\intercal}(Ax^{k+1}+By^{k+1}+Cz^k)+\nabla\psi(y^{k+1})-\nabla\psi(y^k)\\
  &  = B^{\intercal}p^{k+1}+\alpha B^{\intercal}C(z^k-z^{k+1})+\nabla\psi(y^{k+1})-\nabla\psi(y^k),\\
-\nabla h(z^{k+1})& = C^{\intercal}p^{k}+\alpha C^{\intercal}(Ax^{k+1}+By^{k+1}+Cz^{k+1})+\nabla\varphi(z^{k})-\nabla\varphi(z^{k+1})\\
& =
C^{\intercal}p^{k+1}+\nabla\varphi(z^{k})-\nabla\varphi(z^{k+1}).
\end{align*}
 Since
$\|w^k-w^{k+1}\|$ tends to zero, letting $j\to\infty$ in the above
formulas yields
\begin{align*}
A^{\intercal}p^*\in-\partial f(x^*), \ B^{\intercal}p^*\in -\partial g(y^{*}), \\
C^{\intercal}p^*=-\nabla h(z^*), \ Ax^*+By^*+Cz^*=0,
\end{align*}
which implies that $w^*$ is a stationary point of $L_\alpha$.
\end{proof}

\begin{lemma}\label{Le5}
There exists $\kappa>0$ such that for each $k$
\begin{align*}
\mathrm{dist}(0,\partial
\hat{L}(\hat{w}^{k+1}))&\le\kappa(\|x^k-x^{k+1}\|+\|y^k-y^{k+1}\|+\|z^{k}-z^{k+1}\|+\|z^k-x^{k-1}\|).
\end{align*}
\end{lemma}
\begin{proof}
First,  we deduce from  algorithm \eqref{B7} that
\begin{align}
 \partial\hat{L}_x(\hat{w}^{k+1})&=\partial f(x^{k+1})+A^{\intercal}p^{k+1}
   +\alpha A^{\intercal}(Ax^{k+1}+By^{k+1}+Cz^{k+1}),\label{C1}\\
    \partial\hat{L}_y(\hat{w}^{k+1})&=\partial g(y^{k+1})+B^{\intercal}p^{k+1}
   +\alpha B^{\intercal}(Ax^{k+1}+By^{k+1}+Cz^{k+1}),\label{C2}\\
 \partial\hat{L}_z(\hat{w}^{k+1})&=\nabla h(z^{k+1})+C^{\intercal}p^{k+1}
 +\alpha C^{\intercal}(Ax^{k+1}+By^{k+1}+Cz^{k+1})\nonumber \\
 &\quad +2\sigma_0(z^{k+1}-z^{k}), \label{C3}\\
  \partial\hat{L}_{\hat{z}}(\hat{z}^{k+1})&=-\sigma_0(z^{k+1}-z^k),
  \partial\hat{L}_p(\hat{z}^{k+1})=\frac1\alpha(p^{k+1}-p^k).
\end{align}
Second, we apply Fermat's rule to  algorithm \eqref{B7} to get
\begin{align*}
0&\in \partial f(x^{k+1})+A^{\intercal}p^{k}
   +\alpha A^{\intercal}(Ax^{k+1}+By^{k}+Cz^{k})+\nabla\phi(x^{k+1})-\nabla\phi(x^{k}),\\
   0&\in \partial g(y^{k+1})+B^{\intercal}p^{k}
   +\alpha B^{\intercal}(Ax^{k+1}+By^{k+1}+Cz^{k})+\nabla\psi(y^{k+1})-\nabla\psi(y^{k}),
\end{align*}
Substituting this into \eqref{C1} and \eqref{C2}, we obtain
\begin{align*}
 \partial\hat{L}_x(\hat{w}^{k+1})&\ni\alpha A^{\intercal}B(y^{k+1}-y^k)+\alpha A^{\intercal}C(z^{k+1}-z^k)\\
 &\quad+\nabla\phi(x^k)-\nabla\phi(x^{k+1})+A^{\intercal}(p^{k+1}-p^{k}),\\
 \partial\hat{L}_y(\hat{w}^{k+1})
 &\ni\alpha B^{\intercal}C(z^{k+1}-z^k)+B^{\intercal}(p^{k+1}-p^{k})\\
&\quad +\nabla\psi(y^{k})-\nabla\psi(y^{k+1}).
\end{align*}
We also substitute \eqref{D2}  into \eqref{C3} to get
 \begin{align*}
 \partial\hat{L}_z(\hat{w}^{k+1})
=\nabla\varphi(z^{k})-\nabla\varphi(z^{k+1})+C^{\intercal}(p^{k+1}-p^{k})+2\sigma_0(z^{k+1}-z^{k}),
\end{align*}
where the last equality follows from \eqref{B7}.

As $\nabla \phi, \nabla\psi,  \nabla\varphi$ are
 all Lipshitz continuous and matrices $A,B,C$ are all bounded,
 the above series of estimations show that  there exists $\kappa_0>0$ such that
\begin{align}\label{D6}
\mathrm{dist}(0,\partial \hat{L}(\hat{w}^{k+1}))\le
\kappa_0(\|x^k-x^{k+1}\|+ \|y^{k+1}-y^k\|+
\|z^{k+1}-z^k\|+\|p^{k+1}-p^k\|).
\end{align}
On the other hand, it follows from Lemma \ref{Le1} that
\begin{align}
\|p^{k+1}-p^{k}\|
&\le\frac{\sqrt{2}(\ell_{h}+\ell_{\varphi})}{\sqrt{\sigma_C}}
\|z^{k+1}-z^{k}\| +\frac{\sqrt{2}\ell_{\varphi}}{\sqrt{\sigma_C}}\|z^{k}-z^{k-1}\| \\
&\le\frac{\sqrt{2}(\ell_{h}+\ell_{\varphi})}{\sqrt{\sigma_C}}
(\|z^{k+1}-z^{k}\|+\|z^{k}-z^{k-1}\|).
\end{align}
Letting
$\kappa_1:=\sqrt{2}(\ell_{h}+\ell_{\varphi})/\sqrt{\sigma_C},$  we
then have
\begin{align}\label{D3}
 \|p^{k+1}-p^{k}\|\le\kappa_1(\|z^{k+1}-z^{k}\|+\|z^{k}-z^{k-1}\|).
\end{align}
Let $\kappa:= (\kappa_1+1)(\kappa_0+1).$ Hence Lemma \ref{Le5}
follows immediately.
\end{proof}

\subsection{Convergence  analysis}

\begin{theorem}\label{T2} Under the Assumptions 1 and 2, if the sequence $\{u^k\}$ is bounded, then
$$\sum_{k=0}^{\infty}\|w^k-w^{k+1}\|_1<\infty.$$
In particular, the sequence $\{w^k\}$ converges to a stationary
point of $L_{\alpha}$ defined as in \eqref{lg}.
\end{theorem}

\begin{proof}
From the proof of Lemma \ref{Le4}, we see that the sequence
$\{\hat{w}^k\}$ is bounded.
 Let  $\Omega$ stand for the cluster point  set of $\hat{w}^k$. Take
any $\hat{w}^*\in\Omega$ and let $\hat{w}^{k_j}$ be a subsequence of
$\hat{w}^k$  converging to $\hat{w}^*$. Since by Lemma \ref{Le3} the
sequence $\hat{L}(\hat{w}^{k})$ is convergent, it follows that
$$\hat{L}(\hat{w}^{*})=\lim_{j\to\infty}\hat{L}(\hat{w}^{k_j})=
\lim_{k\to\infty}\hat{L}(\hat{w}^{k}),$$ so that  the function
$\hat{L}(\cdot)$ is a constant on $\Omega$.

Let us now consider two possible cases on $\hat{L}(\hat{w}^{k})$.
First assume that there exists $k_0\in\mathrm{N}$ such that
$\hat{L}(\hat{w}^{k_0})=\hat{L}(\hat{w}^{*}).$ Then we deduce from
Lemma
 \ref{Le3}  that for any $k>k_0$
 \begin{align*}
\sigma_1\|w^{k+1}-w^k\|^2&\le
\hat{L}(\hat{w}^{k})-\hat{L}(\hat{w}^{k+1})
 \le \hat{L}(\hat{w}^{k_0})-\hat{L}(\hat{w}^{*})=0,
 \end{align*}
where we have used the fact that $\hat{L}(\hat{w}^{k})$ is
nonincreasing. This together with \eqref{D6}  implies that $(w^k)$
is a constant sequence except for finite terms, and thus the proof
is finished in this case.

Let us now assume that $\hat{L}(\hat{w}^{k})>\hat{L}(\hat{w}^{*})$
for each $k\in\mathrm{N}$. By Assumption 1, It is easy to know that
$\hat{L}(\cdot)$ is a subanalytic function and thus satisfies the
K-L inequality. Then by Lemma \ref{KL} there exists $\eta>0,
\delta>0, \varphi\in\mathscr{A}_{\eta}$, such that
\begin{align*}
 \varphi'(\hat{L}(\hat{w})-\hat{L}(\hat{w}^{*}))\mathrm{dist}(0, \partial \hat{L}(\hat{w}))\ge1.
\end{align*}
 for all
$\hat{w}$ satisfying $\mathrm{dist}(\hat{w},\Omega)<\delta$ and
$\hat{L}(\hat{w}^{*})<\hat{L}(\hat{w})<\hat{L}(\hat{w}^{*})+\eta$.
By definition of $\Omega$ we have $\lim_k\mathrm{dist}(\hat{w}^k,
\Omega)=0.$ This together with the fact
$\hat{L}(\hat{w}^{k})\to\hat{L}(\hat{w}^*)$ implies that there
exists $k_1\in\mathrm{N}$ such that $\mathrm{dist}(\hat{w}^k,
\Omega)<\delta$ and $\hat{L}(\hat{w}^{k})<\hat{L}(\hat{w}^*)+\eta$
for all $k\ge k_1.$

Let us fix $k>k_1$ in the following. Then the K-L inequality
$$\mathrm{dist}(0,\partial\hat{L}(\hat{w}^k))\varphi\prime(\hat{L}(\hat{w}^{k})-\hat{L}(\hat{w}^*))\ge1$$
holds, which along with Lemma \ref{Le5} then yields
\begin{align*}
&\frac{1}{\varphi\prime(\hat{L}(\hat{w}^{k})-\hat{L}(\hat{w}^*))}\le\mathrm{dist}(0,\partial
\hat{L}(\hat{w}^{k+1}))\\
&\le\kappa(\|x^k-x^{k-1}\|+\|y^k-y^{k-1}\|+\|z^k-z^{k-1}\|+\|z^{k-2}-z^{k-1}\|).
\end{align*}
By Lemma \ref{Le2}, the last inequality and  the concavity of
$\varphi$ show
\begin{align*}
&\quad\sigma_1\|w^{k+1}-w^k\|^2\le\hat{L}(\hat{w}^{k})-\hat{L}(\hat{w}^{k+1})\\
&=(\hat{L}(\hat{w}^{k})-\hat{L}(\hat{w}^*))-(\hat{L}(\hat{w}^{k+1})-\hat{L}(\hat{w}^*))\\
&\le\frac{\varphi(\hat{L}(\hat{w}^{k})-\hat{L}(\hat{w}^*))-\varphi(\hat{L}(\hat{w}^{k+1})-\hat{L}(\hat{w}^*))}{\varphi\prime(\hat{L}(\hat{w}^{k})-\hat{L}(\hat{w}^*))}\\
&\le\kappa(\|x^k-x^{k-1}\|+\|y^k-y^{k-1}\|+\|z^k-z^{k-1}\|+\|z^{k-2}-z^{k-1}\|)\\
&\quad\times[\varphi(\hat{L}(\hat{w}^{k})-\hat{L}(\hat{w}^*))-\varphi(\hat{L}(\hat{w}^{k+1})-\hat{L}(\hat{w}^*))],
\end{align*}
or, equivalently,
\begin{align*}
&\quad\|x^{k+1}-x^k\|^2+\|y^{k+1}-y^k\|^2+\|z^{k+1}-z^k\|^2\\
&\le\frac{\kappa}{\sigma_1}(\|x^k-x^{k-1}\|+\|y^k-y^{k-1}\|+\|z^k-z^{k-1}\|+\|z^{k-2}-z^{k-1}\|)\\
&\quad\times[\varphi(\hat{L}(\hat{w}^{k})-\hat{L}(\hat{w}^*))-\varphi(\hat{L}(\hat{w}^{k+1})-\hat{L}(\hat{w}^*))].
\end{align*}
We thus have
\begin{align}\label{D4}
&\quad3(\|x^k-x^{k+1}\|+\|y^k-y^{k+1}\|+\|z^{k+1}-z^k\|)\nonumber\\
&\le3\sqrt{3}(\|x^{k+1}-x^k\|^2+\|y^{k+1}-y^k\|^2+\|z^{k+1}-z^k\|^2)^{1/2}\nonumber\\
&\le2(\|x^k-x^{k-1}\|+\|y^k-y^{k-1}\|+\|z^k-z^{k-1}\|+\|z^{k-2}-z^{k-1}\|)^{1/2}\nonumber\\
&\quad\times\sqrt{\frac{27\kappa}{4\sigma_1}}[\varphi(\hat{L}(\hat{w}^{k})
-\hat{L}(\hat{w}^*))-\varphi(\hat{L}(\hat{w}^{k+1})-\hat{L}(\hat{w}^*))]^{1/2}.
\end{align}
On the other hand, we observe that
\begin{align*}
&\quad \ 2(\|x^k-x^{k-1}\|+\|y^k-y^{k-1}\|+\|z^k-z^{k-1}\|+\|z^{k-2}-z^{k-1}\|)^{1/2}\\
&\quad\times\sqrt{\frac{27\kappa}{4\sigma_1}}[\varphi(\hat{L}(\hat{w}^{k})-\hat{L}(\hat{w}^*))-\varphi(\hat{L}(\hat{w}^{k+1})-\hat{L}(\hat{w}^*))]^{1/2}\\
&\le \|x^k-x^{k-1}\|+\|y^k-y^{k-1}\|+\|z^k-z^{k-1}\|+\|z^{k-2}-z^{k-1}\|\\
&\quad+\frac{27\kappa}{4\sigma_1}[\varphi(\hat{L}(\hat{w}^{k})-\hat{L}(\hat{w}^*))-\varphi(\hat{L}(\hat{w}^{k+1})-\hat{L}(\hat{w}^*))],
\end{align*}
which along with \eqref{D4} yields
\begin{align*}
&\quad \ 3(\|x^k-x^{k+1}\|+\|y^k-y^{k+1}\|+\|z^{k+1}-z^k\|)\\
&\le \|x^k-x^{k-1}\|+\|y^k-y^{k-1}\|+\|z^k-z^{k-1}\|+\|z^{k-2}-z^{k-1}\|\\
&\quad+\frac{27\kappa}{4\sigma_1}[\varphi(\hat{L}(\hat{w}^{k})-\hat{L}(\hat{w}^*))-\varphi(\hat{L}(\hat{w}^{k+1})-\hat{L}(\hat{w}^*))].
\end{align*}
Hence we have
\begin{align*}
&\quad \ \sum_{i=k_1}^{k}3(\|x^i-x^{i+1}\|+\|y^i-y^{i+1}\|+\|z^i-z^{i+1}\|)\\
&\le\sum_{i=k_1}^{k}(\|x^i-x^{i-1}\|+\|y^i-y^{i-1}\|+\|z^i-z^{i-1}\|+\|z^{i-1}-z^{i-2}\|)\\
&\quad+\frac{27\kappa}{4\sigma_1}\sum_{i=k_1}^{k}
[\varphi(\hat{L}(\hat{w}^i)-\hat{L}(\hat{w}^*))-\varphi(\hat{L}(\hat{w}^{i+1})-\hat{L}(\hat{w}^*))].
\end{align*}
Rearranging terms in the above inequality, we obtain
\begin{align*}
&\quad \ 2\sum_{i=k_1}^{k}\|x^i-x^{i+1}\|+2\sum_{i=k_1}^{k}\|y^i-y^{i+1}\|+\sum_{i=k_1}^{k}\|z^i-z^{i+1}\|\\
&\le\sum_{i=k_1}^{k}(\|x^i-x^{i-1}\|-\|x^i-x^{i+1}\|)\\
&\quad+\sum_{i=k_1}^{k}(\|y^i-y^{i-1}\|-\|y^i-y^{i+1}\|)\\
&\quad+\sum_{i=k_1}^{k}(\|z^i-z^{i-1}\|-\|z^i-z^{i+1}\|)\\
&\quad +\sum_{i=k_1}^{k}(\|z^{i-1}-z^{i-2}\|-\|z^i-z^{i+1}\|)\\
&\quad +\frac{27\kappa}{4\sigma_1}\sum_{i=k_1}^{k}
[\varphi(\hat{L}(\hat{w}^i)
-\hat{L}(\hat{w}^*))-\varphi(\hat{L}(\hat{w}^{i+1})-\hat{L}(\hat{w}^*))]\\
&=\|x^{k_1-1}-x^{k_1}\|-\|x^{k}-x^{k+1}\|+\|y^{k_1-1}-y^{k_1}\|-\|y^{k}-y^{k+1}\|\\
&\quad+\|z^{k_1-1}-z^{k_1-2}\|+2\|z^{k_1}-z^{k_1-1}\|-\|z^{k}-z^{k-1}\|-2\|z^{k}-z^{k+1}\|\\
&\quad +\frac{27\kappa}{4\sigma_1}
[\varphi(\hat{L}(\hat{w}^{k_1})-\hat{L}(\hat{w}^*))-\varphi(\hat{L}(\hat{w}^{k+1})-\hat{L}(\hat{w}^*))]\\
&\le\|x^{k_1-1}-x^{k_1}\|+\|y^{k_1-1}-y^{k_1}\|+\|z^{k_1-1}-z^{k_1-2}\|\\
&\quad
+2\|z^{k_1}-z^{k_1-1}\|+\frac{27\kappa}{4\sigma_1}\varphi(\hat{L}(\hat{w}^{0})-\hat{L}(\hat{w}^*))
\end{align*}
where the last inequality follows from the fact that
$\varphi(\hat{L}(\hat{w}^{k+1})-\hat{L}(\hat{w}^*))\ge0.$ Since $k$
is chosen arbitrarily, we deduce that
$\sum_{k=0}^{\infty}(\|x^k-x^{k+1}\|+\|y^k-y^{k+1}\|+\|z^k-z^{k+1}\|)<\infty.$
By inequality \eqref{D3}, it then implies that $\sum_{k=0}^{\infty}
\|p^k-p^{k+1}\|<\infty$, from which
$\sum_{k=0}^{\infty}\|w^k-w^{k+1}\|<\infty$ follows. Consequently
$\{w^k\}$ is a convergent sequence. This completes the proof of
Theorem~\ref{T2}.
\end{proof}

\subsection{Boundedness}

In the previous theorem, we have assumed the boundedness of the
sequence $\{u_k\}$. This assumption is not restrictive in general.
There are actually various sufficient conditions ensuring the
boundedness of the sequence $\{u_k\}$. We present such a sufficient
condition below.

\begin{theorem}\label{T1}
If (a1)-(a3) in Assumption~1 hold and the following (b1)-(b4) are satisfied:
\begin{itemize}
\item[(b1)] $\inf f(x)=f^*>-\infty ,$ $\inf g(y)=g^*>-\infty $ and there exists $%
\beta _{0}>0$ such that $\inf \{h(z)-\beta _{0}\Vert \nabla
h(z)\Vert ^{2}\}=h^*>-\infty ;$

\item[(b2)] $f(x)+g(y)$ is coercive, namely, $\lim_{\min (\Vert x\Vert ,\Vert y\Vert
)\rightarrow \infty }f(x)+g(y)=+\infty$;

\item[(b3)] either $h(z)-\beta _{0}\Vert \nabla h(z)\Vert ^{2}$ is coercive
or $C$ is square;
\item[(b4)] $\alpha > \alpha_0$ where,
\begin{equation*}
\alpha _{0}=\left\{
\begin{array}{cc}
\max \left( \frac{2}{\beta _{0}\sigma _{C}},\frac{4[(\ell _{h}+\ell
_{\varphi })^{2}+\ell _{\varphi }^{2}]}{\mu _{\varphi }\sigma
_{C}}\right) ,
& if ~~h(z)-\beta _{0}\Vert \nabla h(z)\Vert ^{2} \text{~is coercive }\\
\Vert C^{-1}\Vert ^{2}\max \left( \ell _{h} ,\frac{4[(\ell _{h}+\ell
_{\varphi })^{2}+\ell _{\varphi }^{2}]}{\mu _{\varphi }}\right) ,
& if\text{ }C \text{ is square};%
\end{array}
\right.
\end{equation*}
\end{itemize}
then the sequence $\{u^{k}\}$ is bounded.
\end{theorem}

\begin{proof}
First we deduce from Eq. \eqref{B2} that
\begin{equation*}
\frac{1}{\alpha }\Vert p^{k}\Vert ^{2}\leq \frac{2}{\alpha \sigma
_{C}}\Vert \nabla h(z^{k})\Vert ^{2}+\sigma _{0}\Vert
z^{k}-z^{k-1}\Vert ^{2},
\end{equation*}%
which together with the  definition of $\hat{L}$ gets
\begin{align*}
\hat{L}(\hat{w}^{k})& =f(x^{k})+g(y^{k})+h(z^{k})-\frac{1}{\alpha
}\Vert
p^{k}\Vert ^{2}+\sigma _{0}\Vert z^{k}-z^{k-1}\Vert ^{2}+\frac{\alpha }{2}%
\Vert Ax^{k}+By^{k}+Cz^{k}+\frac{p^{k}}{\alpha }\Vert ^{2} \\
& \geq f(x^{k})+g(y^{k})+h(z^{k})-\frac{2}{\alpha \sigma _{C}}\Vert
\nabla
h(z^{k})\Vert ^{2}+\frac{\alpha }{2}\Vert Ax^{k}+By^{k}+Cz^{k}+\frac{p^{k}}{%
\alpha }\Vert ^{2}\\
&\geq f(x^{k})+g(y^{k})+h(z^{k})-\beta _{0}\Vert
\nabla h(z^{k})\Vert ^{2}+\frac{\alpha }{2}\Vert Ax^{k}+By^{k}+Cz^{k}+\frac{p^{k}}{%
\alpha }\Vert ^{2}
\end{align*}
where $\beta _{0}$ is any constant such that $\inf \{h(z)-\beta
_{0}\Vert \nabla h(z)\Vert ^{2}\}>-\infty $ and $h(z)-\beta
_{0}\Vert \nabla h(z)\Vert ^{2}$ keeps coercive no matter whether
$C$ is regular or not. Thus from the
monotonically decreasing property of $\{\hat{L}(\hat{w}^{k})\}$, we obtain%
\begin{equation*}
\hat{L}(\hat{w}^{1})\geq f(x^{k})+g(y^{k})+h(z^{k})-\beta _{0}\Vert
\nabla h(z^{k})\Vert ^{2},
\end{equation*}
which then implies
\begin{equation*}
f(x^{k})+g(y^{k})\leq \hat{L}(\hat{w}^{1})-h^{\ast }<\infty
\end{equation*}%
and
\begin{equation*}
h(z^{k})-\beta _{0}\Vert \nabla h(z^{k})\Vert ^{2}\leq \hat{L}(\hat{w}%
^{1})-f^{\ast }-g^{\ast }<\infty.
\end{equation*}
 By condition (b2), this yields the boundedness of $\{x^{k}\}$ and
$\{y^{k}\},$ and the boundedness of $\{z^{k}\}$ as well whenever
$h(z)-\beta _{0}\Vert \nabla h(z)\Vert ^{2}$ is coercive.

Similarly, from Lemma \ref{Le3}, we can obtain
\begin{equation}\label{A1}
\sigma _{1}\Vert z^{k}-z^{k-1}\Vert ^{2}\leq
\hat{L}(\hat{w}^{1})-(f^{\ast }+g^{\ast }+h^{\ast }):=M_{1}<\infty ,
\end{equation}%
which shows the boundedness of $\{\Vert z^{k}-z^{k-1}\Vert \}.$ Now,
let us assume that the function $h(z)-\beta _{0}\Vert \nabla
h(z)\Vert ^{2}$ is not coercive but the matrix $C$ keeps
nonsingular.
We then justify the boundedness of $\{z^{k}\}$ in this case. In effect, by using again Lemma \ref%
{Le3} and inequality \eqref{A1}, we get
\begin{equation*}
\Vert Ax^{k}+By^{k}+Cz^{k}+\frac{p^{k}}{\alpha }\Vert \leq \sqrt{\frac{%
M_{1}\alpha }{2}},
\end{equation*}%
and using the inequality
\begin{equation*}
\Vert Ax^{k}+By^{k}+Cz^{k}+\frac{p^{k}}{\alpha }\Vert \geq \Vert
Cz^{k}\Vert -\Vert Ax^{k}+By^{k}\Vert -\frac{1}{\alpha }\Vert
p^{k}\Vert ,
\end{equation*}%
we then have
\begin{equation}\label{B4}
\Vert Cz^{k}\Vert -\frac{1}{\alpha }\Vert p^{k}\Vert \leq \sqrt{\frac{%
M_{1}\alpha }{2}}+M_{2},
\end{equation}%
where $M_{2}:=\sup \Vert Ax^{k}+By^{k}\Vert .$ It thus follows from Eq. %
\eqref{D2} and condition (c3) that
\begin{align*}
\Vert p^{k}\Vert & \leq \Vert (C^{\intercal })^{-1}\Vert \Vert
C^{\intercal
}p^{k}\Vert =\Vert C^{-1}\Vert \Vert C^{\intercal }p^{k}\Vert \\
& \leq \Vert C^{-1}\Vert \Vert \nabla h(z^{k})+\nabla \varphi
(z^{k})-\nabla
\varphi (z^{k-1})\Vert \\
& \leq \Vert C^{-1}\Vert (\Vert \nabla h(z^{k})\Vert +\ell _{\varphi
}\Vert z^{k}-z^{k-1}\Vert ).
\end{align*}%
With any fixed $z^{\ast },$ we clearly have
\begin{align*}
\Vert \nabla h(z^{k})\Vert & =\Vert \nabla h(z^{k})-\nabla h(z^{\ast
})\Vert
+\Vert \nabla h(z^{\ast })\Vert \\
& \leq \ell _{h}\Vert z^{k}-z^{\ast }\Vert +\Vert \nabla h(z^{\ast
})\Vert\\
&  \leq \ell _{h}(\Vert z^{k}\Vert +\Vert z^{\ast }\Vert )+\Vert
\nabla h(z^{\ast })\Vert,
\end{align*}%
and furthermore,
\begin{equation*}
\Vert p^{k}\Vert \leq \Vert C^{-1}\Vert \left\{ \ell _{h}(\Vert
z^{k}\Vert +\Vert z^{\ast }\Vert )+\Vert \nabla h(z^{\ast })\Vert
+\ell _{\varphi }\Vert z^{k}-z^{k-1}\Vert \right\} .
\end{equation*}%
Hence we have
\begin{align*}
& \Vert Cz^{k}\Vert -\frac{1}{\alpha }\Vert p^{k}\Vert \geq
\frac{1}{\Vert
C^{-1}\Vert }\Vert z^{k}\Vert -\frac{1}{\alpha }\Vert p^{k}\Vert \\
& \geq \frac{1}{\Vert C^{-1}\Vert }\Vert z^{k}\Vert -\frac{\Vert
C^{-1}\Vert }{\alpha }\left\{ \ell _{h}(\Vert z^{k}\Vert +\Vert
z^{\ast }\Vert )+\Vert \nabla h(z^{\ast })\Vert +\ell _{\varphi
}\Vert z^{k}-z^{k-1}\Vert \right\}
\\
& \geq \left( \frac{1}{\Vert C^{-1}\Vert }-\frac{\Vert C^{-1}\Vert \ell _{h}%
}{\alpha }\right) \Vert z^{k}\Vert -\frac{\Vert C^{-1}\Vert }{\alpha
}(\ell _{h}\Vert z^{\ast }\Vert +\Vert \nabla h(z^{\ast })\Vert
)-\frac{\Vert C^{-1}\Vert \ell _{\varphi }}{\alpha }\Vert
z^{k}-z^{k-1}\Vert ,
\end{align*}%
which together with \eqref{B4} implies
\begin{align*}
\left( \frac{1}{\Vert C^{-1}\Vert }-\frac{\Vert C^{-1}\Vert \ell _{h}}{%
\alpha }\right) \Vert z^{k}\Vert & \leq \sqrt{\frac{M_{1}\alpha }{2}}+M_{2}+%
\frac{\Vert C^{-1}\Vert }{\alpha }(\ell _{h}\Vert z^{\ast }\Vert
+\Vert
\nabla h(z^{\ast })\Vert )+\frac{\Vert C^{-1}\Vert \ell _{\varphi }}{\alpha }%
\Vert z^{k}-z^{k-1}\Vert \\
& \leq \sqrt{\frac{M_{1}\alpha }{2}}+M_{2}+\frac{\Vert C^{-1}\Vert }{\alpha }%
\left( \ell _{h}\Vert z^{\ast }\Vert +\Vert \nabla h(z^{\ast })\Vert
+\ell _{\varphi }\sqrt{\frac{M_{1}}{\sigma _{1}}}\right) ,
\end{align*}%
where the last inequality follows from \eqref{A1}. By condition
(b4), the sequence $\{z^{k}\}$ is then bounded, and so is the
sequence $\{u^{k}\}$.
\end{proof}

\begin{remark}
It is easy to see that function $h(x) = \|Ax - b\|^2$ for any matrix
$A$ and $b$ satisfies conditions (b1) and (b3) with $\beta_0=
\frac{\|A\|^2}{4}$.
\end{remark}

\subsection{Main results}~\label{mainRes}

Combining theorems~\ref{T2} and ~\ref{T1}, we present the following
convergence theorem for the BADMM  with 3-block
procedure~\eqref{B7}.

\begin{theorem}\label{T3}
If Assumption 1 and conditions (b1)-(b4) in Theorem ~\ref{T1} are
satisfied, then the sequence $\{w^k\}$ generated by
procedure~\eqref{B7} converges to a stationary point of $L_{\alpha}$
defined as in \eqref{lg}.
\end{theorem}

We now extend this result to the $N$-block case. Thus, let us
consider the following composite optimization problem:
\begin{align}\label{p2}
\begin{split}
\min & \ f_1(x_1)+f_2(x_2)+\cdots f_N(x_N)  \\
 \mathrm{s.t.} & \  A_1x_1+A_2x_2+\cdots +A_Nx_N=0,
\end{split}
\end{align}
where $A_i\in\mathrm{R}^{m\times n_i}$,
$f_i:\mathrm{R}^{n_i}\to\mathrm{R}, i=1,2,\cdots, N-1$ are proper
lower semicontinuous functions, and
$f_N:\mathrm{R}^{n_N}\to\mathrm{R}$ is a smooth function. The
associated BADMM algorithm takes the form:
 \begin{align}\label{A2}
\begin{split}
\left\{
   \begin{array}{ll}
x_1^{k+1}& =\arg\min\limits_{x_1\in\mathrm{R}^{n_1}} L_{\alpha}(x_1,x_2^{k},\cdots,x_N^{k},p^k)+\triangle_{\phi_1}(x_1,x_1^k)  \\
\quad \vdots &=\qquad \vdots \qquad \vdots  \\
x_N^{k+1}& =\arg\min\limits_{x_N\in\mathrm{R}^{n_N}} L_{\alpha}(x_1^{k+1},\cdots,x_{N-1}^{k+1},x_N,p^k)+\triangle_{\phi_N}(x_N,x_N^k)  \\
p^{k+1} &=p^k+\alpha(A_1x_1^{k+1}+A_2x_2^{k+1}+\cdots+A_Nx_N^{k+1})
 \end{array} \right.
\end{split}
\end{align}
where $\triangle_{\phi_i}, i=1,2,\cdots, N $ are the Bregman
distances associated with  functions $\phi_i$ and the corresponding
Lagrangian function
 $L_{\alpha}:\mathrm{R}^{n_1}\times\mathrm{R}^{n_2}
 \times\cdots\times\mathrm{R}^{n_N}\times\mathrm{R}^{m}\to  \mathrm{R}$ is defined by
\begin{align}\label{lg2}
L_{\alpha}(x_1,x_2\cdots,x_N,p):=\sum_{i=1}^N f_i(x_i)+ \sum_{i=1}^N\langle p,
A_ix_i\rangle+\frac\alpha2\|\sum_{i=1}^N A_ix_i\|^2.
\end{align}
It is then straightforward  to establish a similar convergence
result with Theorem \ref{T3}.

\begin{theorem}\label{T4}
If the following (d1)-(d7) are satisfied:
\begin{itemize}
\item[(d1)] $\langle A_N A_{N}^{\intercal}x, x\rangle = \|x\|_{A_{N}^{\intercal}}^{2} \geq \sigma_{A_N} \| x\|^2,  \forall x \in \mathrm{R}^{n_N}$, namely, $A_N$ is full row rank;
\item[(d2)] $\nabla f_N, \nabla \phi_i, i=1,2,\cdots, N$ are Lipschitz continuous;
\item[(d3)] either $f_i$ or $\phi_i, i=1,2,\cdots, N$ is strongly
  convex;
\item[(d4)] $ f_1+f_2+\cdots +f_{N}$ is subanalytic and coercive;
\item[(d5)] $\inf f_i=f_i^*>-\infty, i=1,2,\cdots, N-1,$ and there exists $%
\beta _{0}>0$ such that $\inf \{f_N(x_N)-\beta _{0}\Vert \nabla
f_N(x_N)\Vert ^{2}\}=f_N^*>-\infty ;$

\item[(d6)]

either $f_N-\beta _{0}\Vert \nabla f_N\Vert ^{2}$ is coercive, or
$A_N$ is square;

\item[(d7)]  $\alpha > \alpha_0$ where,
\begin{equation*}
\alpha _{0}=\left\{
\begin{array}{cc}
\max \left( \frac{2}{\beta _{0}\sigma _{A_N}},\frac{4[(\ell
_{f_N}+\ell _{\phi_N})^{2}+\ell _{\phi_N}^{2}]}{\mu_{N}\sigma
_{A_N}}\right) , & if ~~f_N-\beta _{0}\Vert \nabla
f_N\Vert ^{2} \text{~is coercive, }\\
\Vert A_N^{-1}\Vert ^{2}\max \left( \ell_{f_N},\frac{4[(\ell
_{f_N}+\ell _{\phi_N})^{2}+\ell _{\phi_N}^{2}]}{\mu_{N}}\right),
& \text{if } A_N \text{ is square};%
\end{array}
\right.
\end{equation*}
where $\mu_N$ is the strong convexity coefficient of $f_N$ or
$\varphi_N$, and $\ell _{f_N} $ and $\ell _{\phi_N}$ are
respectively the Lipschitz coefficient of $\nabla f_N$ and $\nabla
\phi_N$,
\end{itemize}
then the sequence $\{x_1^k,x_2^k,\cdots,x_N^k,p^k\}$ converges to a
stationary point of $L_{\alpha}$ defined as in \eqref{lg2}.
\end{theorem}

\begin{remark}
We notice that whenever any $f_i$ is strongly convex, the function
$\phi_i$ in the Bregman distance can be taken as zero in the $i$-th
update of procedure \eqref{A2}.
\end{remark}

\begin{remark} For convenience of applications, we list some
specifications of Theorem \ref{T4} as follows.

(i) {\em Underdetermined linear system of equations:} In this case,
$f_i\equiv0, i=1,2,\cdots,N$, and $m <\sum_{i=1}^N n_i$. The problem
\eqref{p2} is degenerated to
\begin{align}\label{p3}
\begin{split}
 \min & \qquad \qquad 0 \\
 \mathrm{s.t.} & \  A_1x_1+A_2x_2+\cdots + A_Nx_N=0
\end{split}
\end{align}
which amounts  to solving the underdetermined linear system of
equations:
\begin{align}\label{p6}
   Ax=0
\end{align}
where $A=[A_1, A_2,\cdots,A_N]$ and $x=\big[x_{1}^{\intercal},x_{2}^{\intercal},\cdots,x_{N}^{\intercal}\big]^{\intercal}$.
 In this
case, the BADMM algorithm takes the form:
\begin{align}\label{A3}
\begin{split}
\left\{
   \begin{array}{ll}
x_1^{k+1}& =\arg\min\limits_{x_1\in\mathrm{R}^{n_1}}\frac\alpha2\|A_1x_1+A_2x_2^k+\cdots+
A_Nx_N^k+\frac{p^k}{\alpha}\|^2+\triangle_{\phi_1}(x_1,x_1^k)  \\
\quad \vdots &=\qquad \vdots \qquad \vdots  \\
x_N^{k+1}& =\arg\min\limits_{x_N\in\mathrm{R}^{n_N}} \frac\alpha2\|A_1x_1^{k+1}+\cdots+
A_{N-1}x_{N-1}^{k+1}+A_Nx_N+\frac{p^k}{\alpha}\|^2+\triangle_{\phi_N}(x_N,x_N^k)  \\
p^{k+1} &=p^k+\alpha(A_1x_1^{k+1}+A_2x_2^{k+1}+\cdots+A_Nx_N^{k+1}).
 \end{array} \right.
\end{split}
\end{align}

We easily check that in this special case all the assumptions in
Theorem \ref{T4} are met whenever the matrix $A_N$ is nonsingular.
So, by Theorem \ref{T4}, the procedure \eqref{A3} can converge  to a
point $(x_1^*,x_2^*,\cdots,x_N^*,p^*)$. The point
$(x_1^*,x_2^*,\cdots,x_N^*)$  is clearly a solution of \eqref{p6} by
the last equation in \eqref{A3}. We notice that  the same problem
was studied by Sun, Luo and Ye \cite{sly}, and they considered the
case that $A$ is a square nonsingular matrix. To solve the linear
system of equations, they suggested a novel  randomly permuted ADMM
and proved its expected convergence.

 (ii) {\em Two blocks case: $N=2$.} It is easily  seen that Theorem
 \ref{T4} in this case is degenerated to convergence of the conventional BADMM procedure:
 \begin{align}\label{A4}
\begin{split}
\left\{
   \begin{array}{ll}
x_1^{k+1}& =\arg\min\limits_{x_1\in\mathrm{R}^{n_1}} L_{\alpha}(x_1,x_2^{k},p^k)+\triangle_{\phi_1}(x_1,x_1^k)  \\
x_2^{k+1}& =\arg\min\limits_{x_2\in\mathrm{R}^{n_2}} L_{\alpha}(x_1^{k+1},x_2,p^k)+\triangle_{\phi_2}(x_2,x_2^k)  \\
p^{k+1} &=p^k+\alpha(A_1x_1^{k+1}+A_2x_2^{k+1})
 \end{array} \right.
\end{split}
\end{align}
for the problem:
\begin{align}\label{p4}
\begin{split}
 \min & \ f_1(x_1)+f_2(x_2)\\
 \mathrm{s.t.} & \  A_1x_1+A_2x_2=0.
\end{split}
\end{align}
Thus, Theorem \ref{T4} includes the results established in
\cite{lp,wxx} as special cases.

(iii) {\em The unconstrained minimization case:}
\begin{align}\label{p5}
\min f_1(x_1)+f_2(x_2)+\cdots f_N(x_N)
\end{align}
where $f_i:\mathrm{R}^{n_i}\to\mathrm{R}, i=1,2,\cdots, N-1$ are
proper lower semicontinuous functions, and
$f_N:\mathrm{R}^{n_N}\to\mathrm{R}$ is a smooth function. Even no
constraint exists in this case, a similar Bregman alternative
direction method (BADM) can be defined as follows:
\begin{align}\label{A5}
\begin{split}
\left\{
   \begin{array}{ll}
x_1^{k+1}& =\arg\min\limits_{x_1\in\mathrm{R}^{n_1}}f_1(x_1)+\triangle_{\phi_1}(x_1,x_1^k)  \\
\quad \vdots &=\qquad \vdots \qquad \vdots  \\
x_N^{k+1}& =\arg\min\limits_{x_N\in\mathrm{R}^{n_N}} f_N(x_N)+
\triangle_{\phi_N}(x_N,x_N^k).
 \end{array} \right.
\end{split}
\end{align}
Following exactly the procedure of proof of  Theorems \ref{T2} and
\ref{T1}, we can immediately obtain the following  convergence of
 \eqref{A5} in the setting that:
\begin{itemize}
\item[(e1)] $\inf f_i=f_i^*>-\infty, i=1,2,\cdots, N$;

\item[(e2)] $\nabla f_N, \nabla \phi_i, i=1,2,\cdots, N$ are Lipschitz continuous;

\item[(e3)] either $f_i$ or $\phi_i, i=1,2,\cdots, N$ is strongly
  convex;

\item[(e4)] $ f_1+f_2+\cdots +f_{N}$ is subanalytic and coercive.
\end{itemize}
\end{remark}

\section{Demonstration examples}
In this section, a simulated example and a real-world application are provided to support
the correctness of convergence of the proposed 3-block Bregman ADMM for solving non-convex composite problems.

Consider the  non-convex optimization problem with 3-block variables
deduced from matrix decomposition applications~(see e.g.
\cite{br,xcm,zhou2010}):
\begin{equation}
\label{rpca_model}
\begin{split}
&\underset{\mathbf{L},\mathbf{S},\mathbf{T}}{\min} ~\| \mathbf{L}\|_{*} + \lambda \| \mathbf{S}\|_{1/2}^{1/2} + \frac{\mu}{2} \| \mathbf{T}- \mathbf{M}\|_{F}^{2} \\
&s.t.~ \mathbf{T} = \mathbf{L} +\mathbf{S},
\end{split}
\end{equation}
where $\mathbf{M}$, $\mathbf{T}$, $\mathbf{L}$ and $\mathbf{S}$ are all $m\times n$ matrices,
$\mathbf{M}$ is a given observation, $\mathbf{T}$ is an ideal observation, $\|\mathbf{L}\|_{*}:=\sum_{i=1}^{\min(m,n)}\sigma_{i}(\mathbf{L})$
is the nuclear norm of $\mathbf{L}$, $\| \mathbf{S}\|_{1/2}^{1/2}:=\sum_{i=1}^{m}\sum_{j=1}^{n} |\mathbf{S}_{ij}|^{1/2}$ is the $\ell_{1/2}$ quasi-norm of $\mathbf{S}$, $\lambda$ is a trade-off parameter between the spectral sparsity term $\| \mathbf{L}\|_{*}$ and the element-wise sparsity term $\|\mathbf{S}\|_{1/2}^{1/2}$, and $\mu$ is a parameter associated with the noise level. The augmented Lagrange function of this optimization problem is given by
\begin{equation}
L_{\alpha}(\mathbf{L},\mathbf{S},\mathbf{T},\mathbf{\Lambda})= \| \mathbf{L}\|_{*} + \lambda \| \mathbf{S}\|_{1/2}^{1/2} + \frac{\mu}{2} \| \mathbf{T}- \mathbf{M}\|_{F}^{2}+\langle \mathbf{p}, \mathbf{T} - (\mathbf{L}+\mathbf{S}) \rangle + \frac{\alpha}{2}\| \mathbf{T} - (\mathbf{L}+\mathbf{S})\|_{F}^{2}.
\end{equation}
According to the 3-block BADMM~\eqref{B7}, the optimization
problem~(\ref{rpca_model}) can be solved by the following procedure
\begin{align}\label{subProb4}
\begin{split}
\left\{
   \begin{array}{ll}
\mathbf{L}^{k+1} = \underset{\mathbf{L}}{\arg\min}~L_{\alpha}(\mathbf{L},\mathbf{S}^{k},\mathbf{T}^{k},\mathbf{\Lambda}^{k}) + \triangle_{\phi}(\mathbf{L},\mathbf{L}^{k}) \\
\mathbf{S}^{k+1} = \underset{\mathbf{S}}{\arg\min}~L_{\alpha}(\mathbf{L}^{k+1},\mathbf{S},\mathbf{T}^{k},\mathbf{\Lambda}^{k}) + \triangle_{\psi}(\mathbf{S},\mathbf{S}^{k})  \\
\mathbf{T}^{k+1} = \underset{\mathbf{T}}{\arg\min}~L_{\alpha}(\mathbf{L}^{k+1},\mathbf{S}^{k+1},\mathbf{T},\mathbf{\Lambda}^{k}) + \triangle_{\varphi}(\mathbf{T},\mathbf{T}^{k}) \\
\mathbf{p}^{k+1} = \mathbf{p}^{k} + \alpha(\mathbf{T}^{k+1} - (\mathbf{L}^{k+1}+\mathbf{S}^{k+1})).
 \end{array} \right.
\end{split}
\end{align}
Specifying $\phi(\cdot)=\psi(\cdot)= \frac{\gamma_1}{2}\|\cdot\|^2$,
$\varphi(\cdot)=\frac{\gamma_2}{2}\|\cdot\|^2$ and substituting
these formulations into the procedure~\eqref{subProb4}, we then
obtain the following closed-form iterative formulas of
\eqref{subProb4}:
\begin{align}\label{subForm}
\begin{split}
\left\{
   \begin{array}{ll}
\mathbf{L}^{k+1} = \mathcal{S}_{M}(\frac{\alpha(\mathbf{T}^{k}-\mathbf{S}^{k} + \frac{\mathbf{p}^{k}}{\alpha})+\gamma_1\mathbf{L}^{k}}{\alpha + \gamma_1}, ~\frac{\gamma_1}{\alpha+\gamma_1}) \\
\mathbf{S}^{k+1} = \mathcal{H}_{E}(\frac{\alpha(\mathbf{T}^{k}-\mathbf{L}^{k+1} + \frac{\mathbf{p}^{k}}{\alpha})+\gamma_1\mathbf{S}^{k}}{\alpha+\gamma_1}, ~\frac{\lambda}{\alpha+\gamma_1})  \\
\mathbf{T}^{k+1}= \frac{\mu \mathbf{M} + \alpha(\mathbf{L}^{k+1} +\mathbf{S}^{k+1} - \frac{\mathbf{p}^{k}}{\alpha}) + \gamma_2\mathbf{T}^{k}}{\mu + \alpha + \gamma_2} \\
\mathbf{p}^{k+1} = \mathbf{p}^{k} + \alpha(\mathbf{T}^{k+1} - (\mathbf{L}^{k+1}+\mathbf{S}^{k+1}))
 \end{array} \right.
\end{split}
\end{align}
where $\mathcal{S}_{M}(\mathbf{A},~\cdot)$ indicates the operation
of thresholding the singular values of matrix $\mathbf{A}$ using the
well-known soft shrinkage operator, and
$\mathcal{H}_{E}(\mathbf{A},~\cdot)$  the operation of thresholding
the entries of matrix $\mathbf{A}$ using the half shrinkage
operator~\cite{xc,zlx,zfx,zxz}. The procedure~\eqref{subForm} is the
specification  of BADMM \eqref{B7} for the solution of
problem~\eqref{rpca_model} with functions $f(x)$, $g(y)$, $h(z)$
defined by $f(\mathbf{L})=\|\mathbf{L}\|_{*}$,
$g(\mathbf{S})=\lambda\|\mathbf{S}\|_{1/2}^{1/2}$,
$h(\mathbf{T})=\frac{\mu}{2}\|\mathbf{T} - \mathbf{M}\|^2$ and
matrices $A$, $B$, $C$ defined by $A = I$, $B=-I$, $C=-I$ where $I$
is the identity matrix. It is direct to see that all the assumptions
of Theorem~\ref{T3} are satisfied. Consequently, Theorem~\ref{T3}
can be applied to predict
  convergence of~\eqref{subForm} in theory. We conduct a
simulation study and an application example below for support of
such theoretical assertion.

We first expatiate some implementation issues. We set $\gamma_1
=\alpha$ and $\gamma_2 =\alpha + \mu$ in \eqref{subForm}. In order
to avoid the tediousness of tuning the parameter $\alpha$, we
exploit a dynamic updating scheme, e.g., $\alpha = \min( \alpha*1.1,
\alpha_{max})$, where $\alpha_{max}$ is a very large constant. Due
to the non-convexity of this optimization problem it is very
important to choose a suitable initialization. In the following
experiments, we initialized matrix $\mathbf{L}$ by the best rank $r$
approximation of matrix $\mathbf{M}$, i.e.,
$\mathbf{L}=\text{SVD}(\mathbf{M},r)$, where $r$ was empirically set
as $\text{ceil}(0.01\cdot\min(m,n))$; initialized matrix
$\mathbf{S}$ as one zero matrix of size $m \times n$; and then
initialized matrix $\mathbf{T}= \mathbf{L} + \mathbf{S}$. Finally,
we terminated the algorithm by the criterion $\text{relChg} < 1$e-8,
where $\text{relChg}$ is defined as
\begin{equation*}
\text{relChg}:=\frac{\| [\mathbf{L}^{k+1}-\mathbf{L}^{k},\mathbf{S}^{k+1}-\mathbf{S}^{k},\mathbf{T}^{k+1}-\mathbf{T}^{k} ]\|_{F}}{\|[\mathbf{L}^{k},\mathbf{S}^{k},\mathbf{T}^{k}]\|_F +1 }.
\end{equation*}

\begin{figure}[htp]
\centering
\includegraphics[width=5in,height=2.2in]{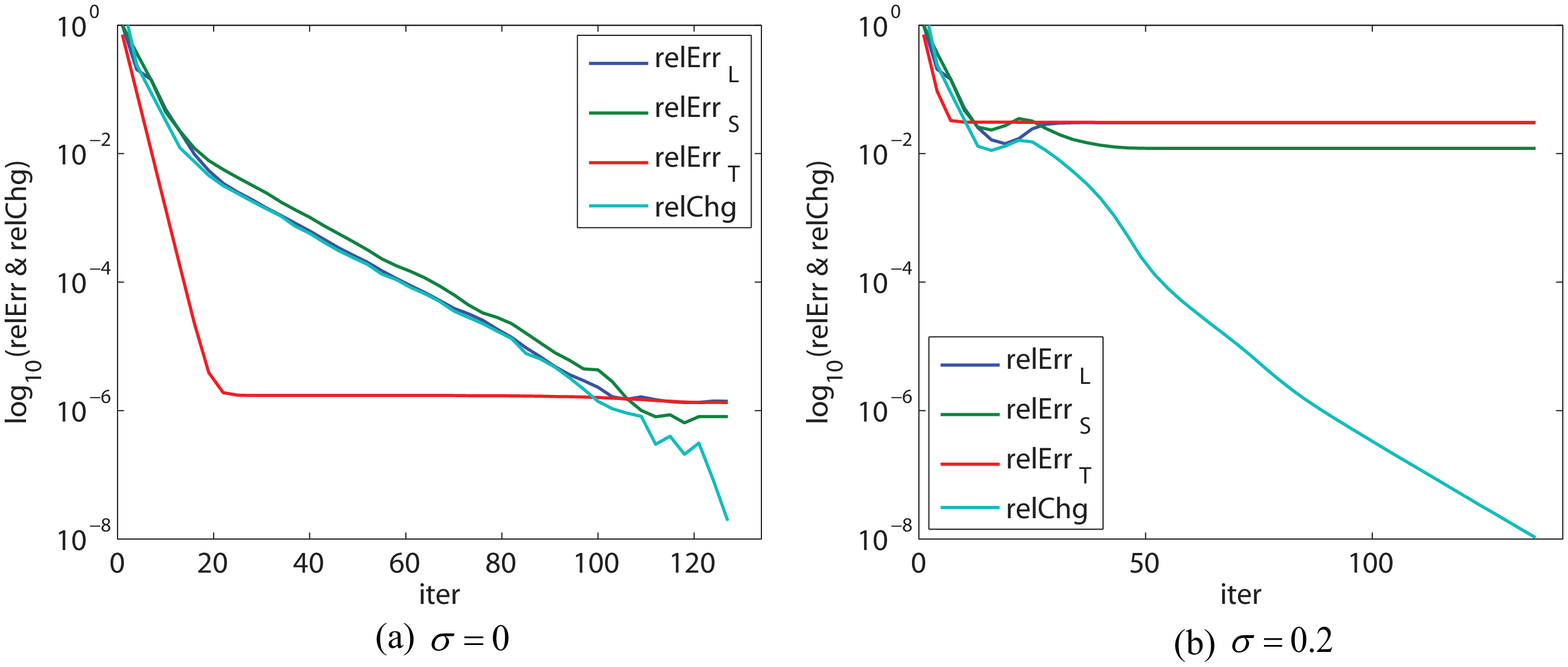}
\caption{Separation results in simulated data.}
\label{simuFig} 
\end{figure}
(a)~\textbf{Simulation study.} To check the validity of
model~\eqref{rpca_model} and the convergence of
procedure~\eqref{subForm}, we generated an observation matrix
$\mathbf{M}$ from given $\mathbf{L}$ and $\mathbf{S}$ (namely, the
true solution) with Gaussian random disturbance $\mathbf{N}$, and
then we applied procedure~\eqref{subForm} to recover $\mathbf{L}$
and $\mathbf{S}$. The square matrices of size $m \times m$ are
randomly generated for our simulations. The matrix $\mathbf{L}$ was
taken as $\mathbf{U}\mathbf{V}^{T}$, where $\mathbf{U}$ and
$\mathbf{V}$ are independent $m\times r$ matrices whose elements are
i.i.d. Gaussian random variables with zero mean and unit variance,
and $\mathbf{S}$ taken as a sparse matrix whose support was chosen
uniformly at random with the entries uniformly specified in the
interval $[-50,50]$. Then, the measurement $\mathbf{M}$ was
generated as $\mathbf{M}=\mathbf{L}+ \mathbf{S} + \mathbf{N}$, where
matrix $\mathbf{N}$ is Gaussian noise with mean zero and variance
$\sigma^2$. Thus, $\sigma=0$ corresponds to the no noise case and
$\sigma \neq 0$ corresponds to the noisy case. In simulations, the
parameter $\mu$ in model~\eqref{rpca_model} was set as a large value
$\text{1e+4}$ in the no noise setting, and a value in the noisy
setting from a candidate set such that the proposed algorithm has
the best performance. The parameter $\lambda$ was empirically set as
the value $\frac{60}{\max(m,n)}$. The performance of the algorithm
is then measured in terms of the relative error defined by
$$\text{relErr}_{\text{\:A}}:=\frac{\| \mathbf{\hat{A}} -\mathbf{A}^{*}\|_{F}}{\| \mathbf{A}^{*}\|_F},$$
where $\mathbf{\hat{A}}$ indicates the recovery result of the algorithm, and $\mathbf{A}^{*}$ indicates the true result.

\begin{figure}[htp]
\centering
\includegraphics[width=5in,height=3in]{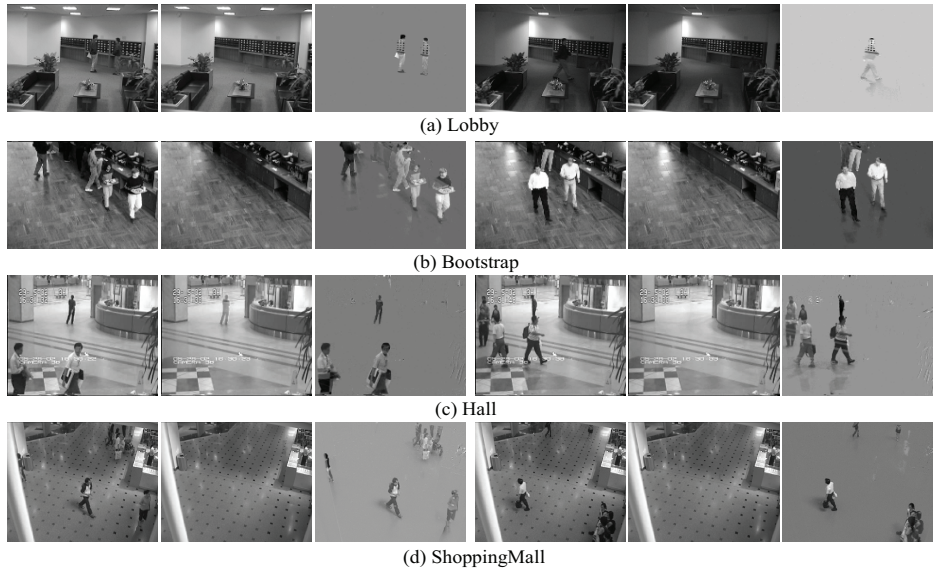}
\caption{Separation results in real-world video clips.}
\label{visualFig}
\end{figure}
With the above settings and measure, our simulation results are then
shown in Figure~\ref{simuFig}. In Figure~\ref{simuFig}(a), they are
exhibited the curves of the relative error
$\text{relErr}_{\mathbf{A}}\: (\mathbf{A}:=\mathbf{L},
\mathbf{S},\mathbf{T})$ and the relative change $\text{relChg}$ with
respect to the iterative steps when no Gaussian noise is added, and
in Figure~\ref{simuFig}(b)   the curves when Gaussian noise is added
with mean 0 and variance $\sigma^2=0.2^2$. From these curves, it can
be seen that under the initialization in terms of the relative error
and the relative change the procedure~\eqref{subForm} does converge,
as predicted.

(b)~\textbf{An application example.} We further applied the
model~\eqref{rpca_model} with BADMM~\eqref{subForm} to the
background subtraction application. Background
subtraction~\cite{Bouwmans2014} is a fundamental task in the field
of video surveillance. Its aim is to subtract the background from a
video clip and meanwhile detect the anomalies (i.e., moving
objects). From the
webpage~\footnote{http://perception.i2r.a-star.edu.sg/bk\_model/bk\_index},
we first download four video clips: Lobby, Bootstrap, Hall, and
ShoppingMall. Then we chose 600 frames from each video clip and
input these 600 frames into our algorithm. The parameter $\lambda$
was set as the value $\frac{50}{\max(m,n)}$. In
Figure~\ref{visualFig}, we exhibit the separation results of some
frames in four video clips. From Figure~\ref{visualFig}, it can be
seen that our algorithm can produce a clean video background and
meanwhile detect a satisfactory video foreground, which supports the
validity and convergence of the proposed BADMM.

\section*{Acknowledgement}
We will thank Dr. Yao Wang for fruitful conversations. This work was partially supported by the
National 973 Programs (Grant No. 2013CB329404),  the Key Program of
National Natural Science Foundation of China (Grant No. 11131006)
and the National Natural Science Foundation of China (Grant No. 111301253).

\end{document}